\documentclass[11pt]{amsart}

\usepackage{amsmath,amssymb,amsthm,mathtools}
\usepackage{hyperref}
\usepackage{aliascnt,cleveref,comment,xcolor}
\usepackage[margin=34mm]{geometry}
\usepackage[utf8]{inputenc}
\usepackage[T1]{fontenc}

\numberwithin{equation}{section}

\newtheorem{theorem}{Theorem}[section]
\newaliascnt{proposition}{theorem}
\newtheorem{proposition}[proposition]{Proposition}
\aliascntresetthe{proposition}
\newaliascnt{lemma}{theorem}
\newtheorem{lemma}[lemma]{Lemma}
\aliascntresetthe{lemma}
\newaliascnt{remark}{theorem}
\theoremstyle{definition} 

\aliascntresetthe{remark}
\newaliascnt{definition}{theorem}
\newtheorem{definition}[definition]{Definition}
\theoremstyle{plain} 

\aliascntresetthe{definition}
\newcommand{\R}{\mathbb{R}}
\newcommand{\Z}{\mathbb{Z}}
\newcommand{\N}{\mathbb{N}}
\newcommand{\T}{\mathbb{T}}

\newcommand{\A}{\mathcal{A}}
\newcommand{\W}{\mathbb{W}}

\newcommand{\eps}{\varepsilon}

\renewcommand{\hat}{\widehat}

\newcommand{\les}{\lesssim}
\newcommand{\less}{\lessapprox}

\newcommand{\maxx}[2]{
	\max_{\substack{#1 \\ #2}}
}

\definecolor{darkred}{rgb}{0.5, 0.0, 0.1}

\definecolor{problemred}{rgb}{0.8,0,0}

\DeclareMathOperator{\supp}{supp}

\title{A Mizohata--Takeuchi-type estimate for the hyperbolic paraboloid}

\author{Jongchon Kim}
\address{Department of Mathematics\\ City University of Hong Kong\\ Hong Kong SAR}
\email{jongckim@cityu.edu.hk}

\author{Liang Wang}
\address{Department of Mathematics\\ City University of Hong Kong\\Hong Kong SAR}
\email{L.Wang@cityu.edu.hk}

\thanks{This work was supported by a grant from the Research Grants Council of the Hong Kong Administrative Region, China (Project No. CityU 11308924). }

\begin{document}
\begin{abstract}
    We prove Mizohata--Takeuchi-type estimates for the Fourier extension operator associated with the truncated hyperbolic paraboloid in $\R^3$. Our main result controls the weighted $L^2$ norm of the extension operator by the $L^2$ mass of the weight over a multiscale family of parallelepipeds arising from $\ell^2$ decoupling for the hyperbolic paraboloid. We deduce this result from a weighted $L^2$ estimate for functions whose Fourier transforms are supported in a small neighborhood of the hyperbolic paraboloid, which we establish using a bilinear refined decoupling estimate.
\end{abstract}
\maketitle

\section{Introduction}
For a function \(g\) on \([-1,1]^2\), let $E$ denote the Fourier extension operator associated with the truncated hyperbolic paraboloid 
$$H=\{ (\xi,\eta,s): (\xi,\eta)\in [-1,1]^2,\, s=\xi\eta\}$$ 
defined by
\[
    Eg(x,t)
    :=
    \int_{[-1,1]^2}
    e^{2\pi i(x_1\xi+x_2\eta+t\xi\eta)}
    g(\xi,\eta)\,d\xi\,d\eta.
\]

Let $B_R\subset\R^3$ denote the ball of radius $R$ centered at the origin, and let $X$ denote the $X$-ray transform, defined by $Xw(\ell) = \int_\ell w$ for each line $\ell\subset \R^3$. We prove Mizohata--Takeuchi-type estimates for the hyperbolic paraboloid.
\begin{theorem}\label{thm:MT0}
For every \(\eps>0\), there exists \(C_\eps>0\) such that
\[
 \int_{B_R}|Eg|^2w
 \leq
 C_\eps R^{1/2+\eps}
 \| Xw \|_\infty \|g\|_2^2
\]
for every \(R\geq1\), every nonnegative weight \(w\), and every
\(g\in L^2([-1,1]^2)\).
\end{theorem}

Weighted estimates for Fourier extension operators provide information not only about the size of the level sets of $Eg$, but also about their spatial concentration. To place this result in context, let $\Sigma\subset\R^n$ be a compact $C^2$ hypersurface equipped with surface measure $d\sigma$. Weighted estimates of the form 
\[   \int_{B_R}
    \bigl|\widehat{gd\sigma}\bigr|^2 w
    \leq C R^\alpha \| Xw \|_\infty 
    \|g\|_{L^2(d\sigma)}^2\]
are referred to as Mizohata--Takeuchi-type estimates. 
The Mizohata--Takeuchi conjecture (see \cite{BaBeCa08} and references therein) asserted that such an estimate should hold with $\alpha=0$, or equivalently, that the corresponding global estimate over $\R^n$ should hold. 
When $\Sigma$ is contained in an affine hyperplane, the conjectured estimate holds for arbitrary nonnegative weights by Plancherel's theorem. It also holds for the sphere under the additional assumption that the weight is radial \cite{BaRuVe97,CaSoVa07}. However, Cairo \cite{Cai25} recently disproved the conjecture for every nonflat $C^2$ hypersurface. When $\Sigma$ is strictly convex and has nonvanishing Gaussian curvature, Carbery, Iliopoulou, and Wang \cite{CaIlWa} established an estimate corresponding to \Cref{thm:MT0}; see the discussion following \Cref{thm:MT} below. We refer to \cite{CaIlWa} and the references therein for further positive results.  

The study of Fourier extension operators originates in the Fourier restriction problem. Motivated by an observation of Stein in the 1960s, 
the restriction problem concerns the Lebesgue space mapping properties of the Fourier restriction operator $f\longmapsto \hat{f}|_\Sigma$. In particular, a restriction estimate shows that, for certain $1<p<2$, the Fourier transform of an $L^p$ function can be meaningfully restricted to the measure-zero set $\Sigma$.
By duality, Fourier restriction estimates for a hypersurface $\Sigma$ as above may be formulated as $L^p(\Sigma,d\sigma)\to L^q(\R^n)$ estimates for the corresponding Fourier extension operator
\[ g \longmapsto \widehat{g d\sigma}. \]
When $\Sigma = H$, this operator can be expressed in terms of $E$ up to multiplication of $g$ by a harmless factor. Consequently, the two operators have equivalent Lebesgue space mapping properties. See \cite{ChoLee17, DW} and the references therein for recent progress on the Fourier restriction problem for $H$. 

We now formulate a refined version of \Cref{thm:MT0}. We first introduce a family of parallelepipeds, arising from $\ell^2$-decoupling inequalities (see \cite{GuMaOh} and \cite{DW}), which covers the (vertical) $R^{-1}$ neighborhood of $H$:
$$N_{R^{-1}}(H) = \bigl\{ (\xi, \eta, s) \in \R^{3}: (\xi,\eta) \in [-1,1]^{2},\ |s-\xi\eta| \leq R^{-1} \bigr\}.$$ 
Let $R\in 4^{\N}$. For each pair $\delta=(\delta_1,\delta_2)$ of dyadic numbers satisfying $$R^{-1}\leq \delta_1,\delta_2\leq 1 \quad\text{and}\quad \delta_1\delta_2=R^{-1},$$ we partition the square $[-1,1]^2$ into a collection $\Theta_\delta$ of axis-parallel rectangles of dimensions $\delta_1\times\delta_2$.
We denote the union of these $O(\log R)$ collections by $$\Theta(R^{-1}) = \bigcup_{\delta_1\delta_2=R^{-1}} \Theta_\delta.$$
More generally, if $4^k\leq R<4^{k+1}$ for some integer $k\geq0$, we replace the condition $\delta_1\delta_2=R^{-1}$ with $\delta_1\delta_2=4^{-k}$.

For each $\theta \in \Theta(R^{-1})$, there exists a parallelepiped $\bar{\theta}$ of dimensions $\delta_1\times \delta_2 \times R^{-1}$ that is tangent to $H$ and contains
\[ N_{R^{-1}}(H_\theta) := N_{R^{-1}}(H) \cap (\theta\times \R).\] 
Let $T_\theta$ denote the parallelepiped dual to $\bar{\theta}$ of dimensions $\delta_1^{-1}\times \delta_2^{-1}\times R$ centered at the origin. A precise description of $\bar{\theta}$ and $T_\theta$ is given in \Cref{sec:wavePacket}.

In what follows, we write $T \parallel T_\theta$ if $T$ is a translate of $T_\theta$. For $\Omega\subset[-1,1]^2$ and a line $\ell$, we write $\ell\perp\Omega$ if there exists $(\xi,\eta)\in\Omega$ such that $\ell$ is parallel to the vector 
\[ n(\xi,\eta) = (-\eta,-\xi, 1) \]
normal to the hyperbolic paraboloid at the point $(\xi,\eta,\xi\eta)$. Let $g_\theta := g1_{\theta}$.  
We have the following refined version of \Cref{thm:MT0}.

\begin{theorem}\label{thm:MT}
For every \(\eps>0\), there exists \(C_\eps>0\) such that
\begin{equation}\label{eq:extension-MTthm}
    \int_{B_R}|Eg|^2 w
    \leq
    C_\eps R^\eps
    \max_{\substack{
        \theta\in\Theta(R^{-1})\\
        g_\theta\neq0}}
    \sup_{T\parallel T_\theta}
    \left(
        \int_Tw^2
    \right)^{1/2}
    \|g\|_2^2
\end{equation}
for every \(R\geq1\), every nonnegative weight \(w\), and every
\(g\in L^2([-1,1]^2)\). In particular, 
\begin{equation}\label{main:eq2thm}
 \int_{B_R}|Eg|^2w
 \leq
 C_\eps R^{1/2+\eps}
 \sup_{\ell\perp\supp g} \left(
        \int_{N_{1}(\ell)} w^2
    \right)^{1/2}
 \|g\|_2^2,
\end{equation}
where $N_{1}(\ell)$ denotes the 1-neighborhood of the line $\ell$.
\end{theorem}

If $g_\theta$ is not identically $0$, then there exists $d\in \theta \cap \supp g$ such that $T_\theta$ is covered by $O(R)$ tubes of dimensions $1\times 1\times R$ whose long axes are parallel to $n(d)$ (see \Cref{prop:lPerp} below). Thus,  \eqref{main:eq2thm} is a consequence of \eqref{eq:extension-MTthm}. 
For nonnegative weights roughly constant at scale 1, in the sense that $w(x)\sim w(y)$ whenever $|x-y|\leq 1$, \eqref{main:eq2thm} yields \Cref{thm:MT0} as 
\[ \sup_{\ell} \int_{N_{1}(\ell)} w^2 \les \| Xw \|_{\infty}^2. \]
We refer to \Cref{sec:proofMT0} for the proof of \Cref{thm:MT0} for general weights.

We compare \Cref{thm:MT} with an earlier result. For compact strictly convex  \(C^2\)
surfaces \(\Sigma\subset\mathbb R^n\)  with nonvanishing Gaussian curvature,  Carbery, Iliopoulou and Wang \cite[Theorem~1.2]{CaIlWa} proved that
\begin{equation}\label{eqn:CIW}
    \int_{B_R}
    \bigl|\widehat{gd\sigma}\bigr|^2w
    \leq C_\eps
    R^\varepsilon
    \sup_{T}
    \left(\int_Tw^{\frac{n+1}{2}}\right)^{\frac{2}{n+1}}
    \|g\|_{L^2(d\sigma)}^2,
\end{equation}
where \(T\) runs over tubes of dimensions
\[
    R^{1/2}\times \cdots \times R^{1/2}\times R
\]
whose long axes are parallel to normal directions over \(\supp g\).
In particular, they obtained
\begin{equation}\label{eqn:CIWpower}
    \int_{B_R}
    \bigl|\widehat{gd\sigma}\bigr|^2w
    \leq C_\eps
    R^{\frac{n-1}{n+1}+\varepsilon}
    \sup_{\ell\perp\supp g}Xw(\ell)\,
    \|g\|_{L^2(d\sigma)}^2.
\end{equation}
Thus, for the hyperbolic paraboloid in dimension \(n=3\), \Cref{thm:MT0} and \Cref{thm:MT} are analogues of \eqref{eqn:CIWpower} and \eqref{eqn:CIW}, respectively. 

Since the truncated hyperbolic paraboloid $H$ is not convex, their result does not apply to $H$. In fact, \eqref{eqn:CIW} fails to hold for $H$, which is related to the fact that $H$ contains line segments.
For \(H\), the supremum must instead be enlarged to include translates of the parallelepipeds \(\{T_\theta\}_{\theta\in\Theta(R^{-1})}\), as in \Cref{thm:MT}. Indeed, for the hyperbolic paraboloid, \eqref{eqn:CIW} would amount to the estimate \eqref{eq:extension-MTthm} with $\Theta(R^{-1})$ replaced by the subcollection $\Theta_{\text{square}}(R^{-1})$ consisting of squares of side length comparable to $R^{-1/2}$.
To see that such a stronger estimate cannot hold, one may choose 
\[ g=1_{[0,1]\times [0,R^{-1}]},\qquad w=1_{[0,c]\times[0,cR]\times[0,cR]} \]
for a small absolute constant $c>0$. In this case, $|Eg|\sim R^{-1}$ on the support of $w$, so the left-hand side of  \eqref{eq:extension-MTthm} is bounded below by an absolute constant. On the other hand, every relevant tube of dimensions $R^{1/2} \times R^{1/2} \times R$ intersects the support of $w$ on a set of volume $O(R^{3/2})$. Since $\| g\|_{L^2}^2 = R^{-1}$, the right-hand side of this hypothetical estimate is $O(R^{-1/4+\eps})$ for arbitrary $\eps>0$, which is a contradiction. 

In view of the Mizohata--Takeuchi conjecture and Cairo's counterexample \cite{Cai25}, it is natural to ask whether the factor $R^{1/2+\eps}$ in \eqref{main:eq2thm} can be improved to the subpolynomial loss $R^\eps$. In a talk \cite{Gu22}, Guth showed that certain decoupling axioms alone cannot yield an estimate without a power loss; see \cite[Section 7]{CaIlWa} for more details. 
The power losses in \eqref{eqn:CIWpower} and \eqref{main:eq2thm} match the barrier described by Guth. More recently, Cairo and Zhang \cite{CaZh25} showed that, for any given $C^k$ hypersurface in $\mathbb{R}^n$, there exists an arbitrarily small \(C^k\) perturbation for which the power loss of
\[
    R^{\frac{n-1}{n-1+k}-\varepsilon}
\]
is necessary. In particular, when \(k=2\), this matches the power loss in \eqref{eqn:CIWpower} up to $\eps$. 
The examples of Cairo and Zhang involve specially perturbed
hypersurfaces and do not yield a corresponding lower bound with a power of $R$ for the hyperbolic paraboloid. Thus, it is still open whether the power \(1/2\) in \eqref{main:eq2thm} can be improved for this particular surface.

Several variants of Mizohata--Takeuchi-type estimates have been studied. Geometric variants include weighted estimates for the truncated cone under a one-dimensional condition on the weight \cite{Ort26} and for well-curved curves via refined decoupling inequalities \cite{CaLiPaYu26}. See also \cite{BeGuNaOl25,BeBeGuNaOl25} for Sobolev and orthonormal-system variants and \cite{Shayya23,Wu25} for $L^2\to L^p(w)$ variants. Related weighted estimates for extension operators associated with hyperboloids under an $\alpha$-dimensional growth condition on the weight were obtained in \cite{BaErHa21}.

We deduce \Cref{thm:MT} from a weighted  $L^2$ estimate for functions $f$ with $\supp{\widehat{f}} \subset N_{R^{-1}}(H)$. For each $\theta\in\Theta(R^{-1})$, we define $f_\theta$ by    $\widehat{f_\theta}=\widehat{f}1_{\theta \times \R}$ and consider the wave packet decomposition
\[
    f_\theta=\sum_{T\in\mathbb T_\theta}f_T,
\]
where $\mathbb T_\theta$ is a lattice of translates of $T_\theta$
covering $\R^3$ with $O(1)$ overlap. Each $f_T$ is essentially localized to $T$
in physical space and to a fixed enlargement of $\bar\theta$
in frequency space. We also define
\[
    \|w\|_{L^2_{\mathrm{avg}}(T)}
    :=\left(\frac1{|T|}\int_{\mathbb R^3}w^2\chi_T\right)^{1/2},
\]
where $\chi_T$ is a nonnegative, $L^\infty$-normalized weight adapted to $T$. Precise definitions of the wave packets and the weights $\chi_T$ are given in \Cref{sec:proofthmweightedL2}.

\begin{theorem}\label{thm:weightedL2}
For every \(\eps>0\), there exists \(C_\eps>0\) such that, for every
\(R\geq1\), every nonnegative weight \(w\), and every \(f\) satisfying
\(\supp\widehat f\subset N_{R^{-1}}(H)\), we have
\begin{equation}\label{eqn:mainL2}
    \int_{\mathbb R^3}|f|^2w\leq C_\eps R^\eps\Bigg(\sum_{\theta\in\Theta(R^{-1})}
        \sum_{T\in\mathbb T_\theta}\|f_T\|_2^2\|w\|_{L^2_{\mathrm{avg}}(T)}^2\Bigg)^{1/2}\|f\|_2.
\end{equation}
Consequently, 
\begin{equation}\label{eqn:equivL2}
    \int_{\mathbb R^3}|f|^2w
    \leq
    C_\eps R^\eps
    \max_{\substack{
        \theta\in\Theta(R^{-1})\\
        f_\theta\neq0}}
    \sup_{T\parallel T_\theta}
    \left(
        \frac1{|T|}
        \int_Tw^2
    \right)^{1/2}
    \|f\|_{L^2}^2.
\end{equation}
\end{theorem}

Estimates analogous to  \eqref{eqn:mainL2} were established in the strictly convex and well-curved settings in \cite[Theorem 4.1]{CaIlWa} and \cite[Theorem 2.2]{CaLiPaYu26}, respectively. 
In those settings, the estimates are equivalent to refined decoupling inequalities. The formulation of \eqref{eqn:mainL2}, however, is weaker in that the sum over $\theta$ involves $O(\log R)$ distinct partitions of the support of $\widehat{f}$.  Consequently,  \eqref{eqn:mainL2} does not appear to imply a  refined decoupling inequality for the hyperbolic paraboloid.

We now explain how \eqref{eqn:equivL2} follows from \eqref{eqn:mainL2}. First, we note that if $f_\theta=0$ for some $\theta\in \Theta(R^{-1})$, then $f_T=0$ for all $T\in \T_\theta$. Thus, the sum over $\theta$ in \eqref{eqn:mainL2} may be restricted to those $\theta$ for which $f_\theta\neq 0$. Moreover, the decay of $\chi_T$ gives
\[
    \|w\|_{L^2_{\mathrm{avg}}(T)}^2
    \lesssim
    \sup_{T'\parallel T_\theta}
    \frac1{|T'|}\int_{T'} w^2
\]
for every $T\in\mathbb T_\theta$. 
Applying Plancherel's theorem for each partition $\Theta_\delta$ and recalling that there are $O(\log R)$ admissible pairs $\delta$, we obtain
\[
    \sum_{\theta\in\Theta(R^{-1})}
    \sum_{T\in\mathbb T_\theta}\|f_T\|_2^2 \sim \sum_{\theta\in\Theta(R^{-1})}
    \|f_\theta\|_2^2
    \lesssim
    \log R\|f\|_2^2.
\]
Therefore, after absorbing the logarithmic loss into $R^\eps$, \eqref{eqn:equivL2} follows from \eqref{eqn:mainL2}.

We note that \eqref{eqn:equivL2} is equivalent to \eqref{eq:extension-MTthm} in \Cref{thm:MT}; see \Cref{sec:extension-equivalence}. Estimates of the form \eqref{eqn:equivL2} were developed in connection with the Falconer distance set problem in \cite{GIOW}. See also \cite[Theorem 2.5]{Kim25} for analogous estimates for paraboloids in $\R^n$. 
The exponent $q=2$ in $\left(|T|^{-1} \int_Tw^q\right)^{1/q}$ in \eqref{eqn:equivL2} is sharp, in the sense that it cannot be replaced by any $q<2$. This can be seen by choosing $\widehat{f}$ and $w$ as indicator functions of $N_{R^{-1}}(H)$ and the unit ball $B(0,1)$, respectively. On the other hand, as observed in \cite{CaIlWa}, the Knapp example shows that the exponent $q=2$ in the factor $|T|^{-1/2} \left(\int_T w^q \right)^{1/q}$ in \eqref{eqn:equivL2} cannot be replaced by any larger exponent.

A new difficulty in establishing \eqref{eqn:mainL2}, compared with the setting of \cite{CaIlWa}, is that the direct analogue of the refined decoupling inequality for strictly convex hypersurfaces \cite{GIOW} (see also \cite{du2023weighted}) fails for the hyperbolic paraboloid. This difficulty is already present in decoupling estimates for functions $f$ satisfying $\supp\widehat{f}\subset N_{R^{-1}}(H)$.
Bourgain and Demeter \cite{BoDe17} proved the $\ell^2$-decoupling inequality
\[ \| f\|_{L^p(\R^3)} \leq C_\eps R^{\frac{1}{4}-\frac{1}{2p}+\eps} \bigg(\sum_{\theta \in \Theta_{\text{square}}(R^{-1})} \| f_\theta\|_{L^p}^2\bigg)^{1/2}, \qquad 2\leq p\leq 6.\]
Unlike in the case of the paraboloid (cf. \cite{BoDe15}), the power loss in $R$ is necessary even for $2<p\leq 4$. This is again related to the fact that $H$ contains line segments. On the other hand, Guth, Maldague and Oh \cite{GuMaOh} observed that an $\ell^2$-decoupling inequality with only an $R^\eps$ loss can be obtained by replacing the collection $\Theta_{\text{square}}(R^{-1})$  of squares with the multiscale collection $\Theta(R^{-1})$ of rectangles: 
\[ \| f\|_{L^p(\R^3)} \leq C_\eps R^{\eps} \bigg(\sum_{\theta \in \Theta(R^{-1})} \| f_\theta\|_{L^p}^2\bigg)^{1/2}, \qquad 2\leq p\leq 4.\]
Demeter and Wu \cite{DW} recovered this result as a consequence of a bilinear decoupling estimate for $H$. They also proved a bilinear refined decoupling inequality for $H$ and used it to establish Fourier restriction  estimates.

Our proof of \Cref{thm:weightedL2} combines a broad--narrow analysis (cf. \cite{ChoLee17,DW}) with the bilinear refined decoupling inequality of Demeter and Wu. Specifically, fix a dyadic $K\sim R^{\eps/12}$.  For each $x$, we are in the narrow case if either \[ |f(x)|\les |f_\tau(x)| \] for some $\tau\in\Theta_{\mathrm{square}}(K^{-2})$, or \[ |f(x)|\les |f_S(x)| \] for some horizontal or vertical strip $S$ of dimensions $1\times K^{-1}$. We handle the narrow case by rescaling and induction on scale. Otherwise, we are in the broad case, in which \[ |f(x)|\les K^2 |f_{\tau_1}(x)f_{\tau_2}(x)|^{1/2} \] for some ``transverse'' squares $\tau_1,\tau_2\in\Theta_{\mathrm{square}}(K^{-2})$. This case is handled using the bilinear refined decoupling inequality. Our reliance on this inequality currently restricts the result to dimension three.

We organize the paper as follows. In \Cref{sec:wavePacket}, we define the family $\{T_\theta\}_{\theta\in \Theta(R^{-1})}$ and obtain the covering property used to deduce \eqref{main:eq2thm} from \eqref{eq:extension-MTthm}. In \Cref{sec:proofthmweightedL2}, we reduce the proof of \Cref{thm:weightedL2} to a bilinear estimate, which is established in \Cref{sec:bilinear}. 
Finally, in \Cref{sec:extension-equivalence}, we establish the equivalence between \eqref{eqn:equivL2} and \eqref{eq:extension-MTthm} and then deduce \Cref{thm:MT0}.

We write $A\les B$ if $A\leq CB$ for some absolute constant $C>0$, and $A\sim B$ if $A\les B$ and $B\les A$. We write $A\les_\eps B$ if the implicit constant is allowed to depend on $\eps$. We often write $\| f\|_p$ for the $L^p$-norm $\| f\|_{L^p(\R^3)}$.

\subsection*{Acknowledgements}
JK thanks Zane Li for posing a question at \emph{Harmonic Analysis and its Interactions @ Nagoya 2026} that motivated this work, and the organizers for their warm hospitality and for bringing the participants together.
We thank Po-Lam Yung for asking whether \eqref{eqn:mainL2} was available, which prompted us to strengthen \eqref{eqn:equivL2}. We acknowledge the use of AI tools to assist in improving the language and clarity of the manuscript and identifying minor misprints.

\section{Caps covering the hyperbolic paraboloid}\label{sec:wavePacket}
In this section, we give precise definitions of  $\bar{\theta}$ and $T_\theta$. We choose our caps to be compatible with the affine transformations that will be used in the rescaling arguments (cf. \cite{GuoLiYung21}).

Given a pair of dyadic numbers $0<\delta_1,\delta_2 \leq 1$ and a point $c=(c_1,c_2) \in [-1,1)^2$,  we define the affine map $\mathcal{A}_{c,\delta}:\R^3 \to \R^3$ by 
\[ \mathcal{A}_{c,\delta} (\xi,\eta, s)  = (c_1+  \delta_1\xi ,c_2+\delta_2 \eta, c_1c_2+ c_2 \delta_1 \xi +c_1 \delta_2 \eta + \delta_1\delta_2s ). \] 
This affine transform maps the hyperbolic paraboloid over $[0,1]^{2}$ onto the hyperbolic paraboloid over $c+[0,\delta_1]\times [0,\delta_2]$ since 
\[ \mathcal{A}_{c,\delta} (\xi,\eta, \xi\eta) = \big(c_1+\delta_1 \xi, c_2+\delta_2 \eta, (c_1+\delta_1 \xi)(c_2+\delta_2 \eta) \big). \]

We define
\begin{align*}
		{\theta} = {\theta}(c,\delta) &:=  c+[0,\delta_1]\times [0,\delta_2], \\
	 \bar{\theta} = 	\bar{\theta}(c,\delta) &:=  \mathcal{A}_{c,\delta} \big([0,1]^{2}\times [-2, 2]\big).
\end{align*}
Then $\bar{\theta}(c,\delta)$, called a $\delta$-cap, 
is a parallelepiped with side lengths comparable to $\delta_1$, $\delta_2$, and $\delta_1\delta_2$, tangent to $H$ in the sense that its two tangential edge directions are orthogonal to the vector $n(c)=(-c_2,-c_1,1)$ normal to $H$ at $(c_1,c_2,c_1c_2)$.
 Recall that $\Theta(R^{-1}) = \bigcup_{\delta_1\delta_2=R^{-1}} \Theta_\delta$, where 
$$\Theta_\delta = \{ \theta(c,\delta):c\in 
(\delta_1 \Z \times \delta_2 \Z)\cap [-1,1)^2 \}.$$ 

We now define the parallelepiped  $T_\theta$ dual to $\bar{\theta}$.  For $\theta= \theta(c,\delta)$, we write $\mathcal{A}_\theta = \mathcal{A}_{c,\delta}$.
The standard matrix of the linear part of the affine transform $\mathcal{A}_\theta$ is 
$$M_\theta = M_{c,\delta}= \begin{pmatrix} \delta_1 & 0&0 \\ 0&\delta_2 & 0 \\  c_2 \delta_1& c_1  \delta_2& \delta_1\delta_2 \end{pmatrix}.$$ 
We denote the inverse transpose of $M_\theta$ by $L_\theta=L_{c,\delta} := (M_{c,\delta})^{-T}$. A direct computation shows that 
\[ L_\theta = L_{c,\delta} =  [\delta_1^{-1}e_1, \delta_2^{-1}e_2, (\delta_1\delta_2)^{-1} n(c)].\]
We define $$T_\theta := L_\theta (Q_1), \;\text{where } \; Q_r := [-r,r]^3.$$ 
Then 
\begin{align}\label{eqn:defT}
	T_{\theta} =\{ \delta_1^{-1} u_1 e_1 + \delta_2^{-1} u_2 e_2 + (\delta_1\delta_2)^{-1} u_3 n(c(\theta))\, : \, |u_i|\leq 1 \},
\end{align}
where we write $c(\theta) = c$ if $\theta=\theta(c,\delta)$ for some $\delta$.
Thus, $T_\theta$ is a parallelepiped of dimensions $\delta_1^{-1}\times \delta_2^{-1} \times (\delta_1\delta_2)^{-1}$, whose longest axis is parallel to $n(c(\theta))$. 

It is useful to observe that affine transforms $\A_{c,\delta}$ are well behaved under composition. For $\delta=(\delta_1,\delta_2)\in (0,\infty)^2$ and $c=(c_1,c_2)\in \R^2$, we define 
\[ \delta \circ c := (\delta_1c_1, \delta_2 c_2)    \;\; \text{ and } \;\; \delta^{-1} := (\delta_1^{-1},\delta_2^{-1}).   \]
A direct computation yields, for $c,d\in \R^2$ and $\delta,\Delta \in (0,\infty)^2$, 
\begin{equation}\label{eqn:composition}
    \mathcal{A}_{c,\delta} \circ \mathcal{A}_{d,\Delta} = \mathcal{A}_{ c+ \delta\circ d, \delta\circ \Delta}.
\end{equation} 

\begin{lemma}\label{lem:composition}
Suppose that $\tau = \theta(c,\delta) \in \Theta_\delta$ and $\theta' = \theta(d,\Delta)\in \Theta_\Delta$. Then 
\[ \mathcal A_\tau\circ \mathcal A_{\theta'} = \mathcal A_{\theta}  \]
for $\theta = \theta(c+\delta\circ d, \delta\circ \Delta)$. Consequently, $L_{\tau} L_{\theta'} = L_\theta$ and $$L_\tau(T_{\theta'}) = T_\theta.$$
Moreover, 
\[ \A_\tau\big(N_{(\delta_1 \delta_2 R)^{-1}}(H_{\theta'})\big) = N_{R^{-1}}(H_{\theta}) \]
and 
\[ \A_\tau\big(\theta' \times \R \big) = \theta \times \R. \]
\end{lemma}
\begin{proof}
The identity $\mathcal A_\tau\circ \mathcal A_{\theta'} = \mathcal A_{\theta}$ is just a restatement of \eqref{eqn:composition}. In particular, this yields 
\[ M_\tau M_{\theta'} = M_{\theta}. \]
Taking inverse transpose, we obtain $L_{\tau} L_{\theta'} = L_\theta$. 
Applying both sides to $Q_1$, we obtain $L_\tau(T_{\theta'}) = T_\theta$. 

For $(\xi,\eta,s)\in\R^3$, write $(u,v,w)=\A_\tau(\xi,\eta,s)$. Then $u=c_1+\delta_1\xi$, $v=c_2+\delta_2 \eta$ and 
\[ w= uv +  \delta_1\delta_2 (s-\xi\eta).\]
Therefore,
\[ (\xi,\eta) \in \theta' \text{ and } |s- \xi\eta| \leq (\delta_1\delta_2 R)^{-1},\]
if and only if 
\[ (u,v) \in \theta \text{ and } |w- uv| \leq R^{-1}. \]
Thus, $\A_\tau\big(N_{(\delta_1 \delta_2 R)^{-1}}(H_{\theta'})\big) = N_{R^{-1}}(H_{\theta})$. 

The identity $\A_\tau\big(\theta' \times \R \big) = \theta \times \R$ follows from the first two components of $\A_\tau$.
\end{proof}

The following observation verifies that \eqref{eq:extension-MTthm} implies \eqref{main:eq2thm} in \Cref{thm:MT}.
\begin{proposition}\label{prop:lPerp} Let $g$ be a function supported on $[-1,1]^2$. Then 
    $$\maxx{\theta \in \Theta(R^{-1})}{g_\theta \neq 0} \sup_{T \parallel T_\theta } \int_T w^2 \les R \sup_{\ell\perp \supp{g}} \int_{N_1(\ell)} w^2,$$
    where $N_1(\ell)$ denotes the 1-neighborhood of the line $\ell$. 
\end{proposition}
\begin{proof}
We need the following lemma.
\begin{lemma}\label{lem:inclusion}
	If $d\in \theta(c,\delta)$, then $T_{\theta(c,\delta)} \subset 2T_{\theta(d,\delta)}$.
\end{lemma}
We defer the proof of \Cref{lem:inclusion} and proceed to the proof of \Cref{prop:lPerp}. 

Let $T$ be a translate of $T_\theta$ for some $\theta=\theta(c,\delta) \in \Theta(R^{-1})$ such that $g_\theta \neq 0$. Fix $d\in \theta\cap \supp g$.  By \Cref{lem:inclusion}, $T_\theta \subset 2T_{\theta(d,\delta)}$. Since $T_{\theta(d,\delta)}$ is a parallelepiped of dimensions $\delta_1^{-1}\times \delta_2^{-1}\times R$ with its long axis parallel to $n(d)$ and $\delta_1^{-1}\delta_2^{-1} = R$, we may cover $2T_{\theta(d,\delta)}$, hence $T_\theta$, by $O(R)$ tubes of dimensions $1\times 1\times R$ parallel to $n(d)$. 
Consequently, there are $O(R)$ lines $\ell$ parallel to $n(d)$ (and hence satisfying $\ell\perp\supp g$) such that the corresponding sets $N_1(\ell)$ cover $T$.
This yields \Cref{prop:lPerp}. 
\end{proof}

\begin{proof}[Proof of \Cref{lem:inclusion}]
	Let $c=(c_1,c_2) \in [-1,1)^2$ and $\delta=(\delta_1,\delta_2)$, so that $T_{\theta(c,\delta)}= L_{c,\delta}(Q_1)$. It suffices to show that if $d=(d_1,d_2) \in c+[0,\delta_1]\times [0,\delta_2]$, then $L_{c,\delta}(Q_1) \subset L_{d,\delta}(Q_2)$. 

    By \eqref{eqn:defT}, each element  $x\in L_{c,\delta}(Q_1)$ can be written as
    $$x=\big(\delta_1^{-1}(u_1 - \delta_2^{-1}u_3 c_2), \delta_2^{-1}(u_2 - \delta_1^{-1}u_3 c_1), \delta_1^{-1}\delta_2^{-1}u_3 \big)$$ for some $|u_i|\leq 1$. We set $v_3=u_3$ and 
	\begin{align*}
		v_1 &= u_1 - \delta_2^{-1} u_3(c_2-d_2) \\
		v_2 &= u_2 - \delta_1^{-1} u_3(c_1-d_1),
	\end{align*}
	so that we may write
	$$x=\big(\delta_1^{-1}(v_1 - \delta_2^{-1}v_3 d_2), \delta_2^{-1}(v_2 - \delta_1^{-1}v_3 d_1), \delta_1^{-1}\delta_2^{-1}v_3 \big).$$ 
	Therefore, to verify that $x\in L_{d,\delta}(Q_2)$, it suffices to show that $|v_i| \leq 2$. Clearly, $|v_3| = |u_3| \leq 1$. Since $|c_2-d_2| \leq \delta_2$, 
	\[ |v_1| \leq |u_1| +\delta_2^{-1} |u_3|  |c_2-d_2| \leq   2. \]
	A similar  bound holds for $|v_2|$.
\end{proof}

\section{Proof of \Cref{thm:weightedL2}}\label{sec:proofthmweightedL2}

We shall deduce \Cref{thm:weightedL2} from a refined weighted $L^2$ estimate, which involves wave packets. Before stating the estimate, we fix a standard wave packet decomposition of $f_\theta$ for each $\theta \in \Theta(R^{-1})$. 

\subsection{Wave packet decomposition}
Recall that $Q_r = [-r,r]^3$.
Let $\psi$ be a Schwartz function such that $\hat{\psi} \geq 0$ is supported in $Q_{1/2}$ and $\hat{\psi}=1$ on $Q_{1/4}$. Set $\psi_k = \psi(\cdot - k)$ so that $\hat{\psi_k}(\xi) =e^{-2\pi i k\cdot \xi} \hat{\psi}(\xi)$. For any $L^2$ function $F$ with $\supp \hat{F} \subset Q_{1/4}$, we have $\hat{F} = \hat{F} \hat{\psi}$.  By the Fourier series expansion on $Q_{1/2}$, 
\[ \hat{F}(\xi) = \sum_{k\in \Z^3} a_k e^{-2\pi i k \cdot \xi} \hat{\psi}(\xi) = \sum_{k\in \Z^3} a_k \hat{\psi_k}(\xi), \]
where $a_k$ is the Fourier coefficient of $\hat{F}$ on $Q_{1/2}$. 
By the definition of the Fourier coefficients and Plancherel's theorem,  we may write $a_k$ as 
\[ a_k = \int_{Q_{1/2}} \hat{F}(\xi) e^{2\pi i k\cdot \xi} d\xi = \int \hat{F}(\xi)   \hat{\psi}(\xi)  e^{2\pi i k\cdot \xi} d\xi  = \langle \hat{F}, \hat{\psi_k} \rangle = \langle F, \psi_k \rangle, \]
where $\langle f, g \rangle = \int f \bar g$ denotes the inner product on $L^2$. By  Fourier inversion, we have
\begin{equation}\label{eqn:wavepacketUnitScale}
   F = \sum_{k\in \Z^3} \langle F, \psi_k \rangle \psi_k. 
\end{equation}

For $\theta = \theta(c,\delta) \in \Theta(R^{-1})$, recall that 
$\supp \hat{f_\theta} \subset \bar{\theta} = \A_\theta([0,1]^2\times [-2,2] )$. We set 
\[ \hat{F}(\xi) = J_\theta^{1/2} \hat{f_\theta}( \A_\theta(8 \xi)), \]
where $J_\theta = 8^{3}\delta_1^2\delta_2^2$ is the Jacobian determinant of the mapping $\xi \mapsto \A_\theta (8\xi)$. Then $\supp{\hat{F}} \subset Q_{1/4}$ and 
\[ \| F \|_{L^2} = \| f_\theta \|_{L^2}. \]
A direct computation yields 
\[ \langle F, \psi_k \rangle = \langle f_\theta, \psi_{\theta,k} \rangle,\] where, with $c_\theta = (c_1,c_2,c_1c_2)$,  
\[ \psi_{\theta,k}(x) = J_\theta^{1/2} e^{2\pi i c_\theta \cdot x} \psi(8L_{\theta}^{-1} x - k). \]
Applying the decomposition \eqref{eqn:wavepacketUnitScale} and changing variables, we obtain 
\[ f_\theta = \sum_{k\in \Z^3} \langle f_\theta, \psi_{\theta,k} \rangle \psi_{\theta,k}.  \]

We set 
\[ \T_\theta = \{ T_{\theta,k}: k\in \Z^3 \}, \text{ where } T_{\theta,k} = L_\theta( k/8 + Q_1 ).\]
For $T=T_{\theta,k}\in \T_\theta$, we write $\psi_T = \psi_{\theta,k}$ and set
\[ f_{T} = \langle f_\theta, \psi_{T} \rangle \psi_{T},  \]
so that 
\[ f_\theta = \sum_{T\in \T_\theta} f_T.  \]
Note that $f_T$ is a wave packet, essentially localized to $T$ in physical space, while $\widehat{f_T}$ is supported in the enlarged cap $\A_\theta(Q_4)$, which contains $\bar{\theta}$ and has comparable dimensions.

We record a few standard properties of the wave packets that will be used below. First, observe that $\|\psi_{T} \|_{L^2} = \| \psi \|_{L^2} \sim 1$. Consequently, 
\begin{equation}\label{eqn:coefficient}
    |\langle f_\theta, \psi_{T} \rangle| \sim \| f_T \|_{L^2}. 
\end{equation}
Thus,  by Parseval's identity,
\begin{equation}\label{eqn:orthogonalWavePacket}
\sum_{T\in \T_\theta} \| f_T\|_{L^2}^2 \sim \sum_{T\in \T_\theta} |\langle f_\theta, \psi_{T} \rangle|^2  = \| f_\theta \|_{L^2}^2.
\end{equation}
For each $T=T_{\theta,k} \in \T_\theta$, we fix a $L^\infty$-normalized weight  
\[ \chi_T(x) := (1+|8L_\theta^{-1} x- k |)^{-1000}\] 
adapted to $T$. Since $\psi$ is Schwartz and $J_\theta \sim |T|^{-1}$, we have \[|\psi_T(x)| = J_\theta^{1/2} |\psi(8L_\theta^{-1} x- k )| \les_N |T|^{-1/2} \chi_T^N(x).\] 
Together with \eqref{eqn:coefficient}, this gives
\begin{equation}\label{eqn:ptWavePacket}
    |f_T(x)| \les_N \|f_T\|_{L^2}  |T|^{-1/2} \chi_T^N(x). 
\end{equation}

By \eqref{eqn:ptWavePacket} and Cauchy--Schwarz, we have 
\begin{equation}\label{eq:separated-packet-pointwise}
    \bigg|\sum_{T\in\mathbb T}f_T(x)\bigg|^2
    \lesssim_N
    \bigg(\sum_{T\in \T} \chi_T^N(x) \bigg)\Bigg(\sum_{T\in\mathbb T}
    \frac{\|f_T\|_2^2}{|T|}\chi_T(x)\Bigg)
\end{equation}
for any $\T\subset \cup_{\theta \in \Theta_{\mathrm{square}}(R^{-1})} \T_\theta$. We have the trivial bound
\begin{equation}\label{eqn:trivialBound}
\| \sum_{T\in \T} \chi_T \|_{L^\infty} \les R
\end{equation}
because 
\[ \sum_{T\in \T} \chi_T =\sum_{\theta \in \Theta_{\mathrm{square}}(R^{-1})}  \sum_{T\in \T\cap \T_\theta} \chi_T \les \sum_{\theta \in \Theta_{\mathrm{square}}(R^{-1})} 1 \les R. \]

For each $T_{\theta,k}$, define, for a fixed constant $\eps_0>0$ sufficiently small relative to the underlying $\eps>0$, 
\[ 2T_{\theta,k} := L_\theta (k/8 + Q_{R^{\eps_0}}). \]
Thus, $2T$ is the concentric $R^{\eps_0}$-dilate of $T$ and $\chi_T$ decays sharply away from $2T$. Consequently, if $X\subset \R^3$ is a set such that 
\[ X \cap  2T =\emptyset, \;\; \text{ for all } T\in \T, \]
then by choosing $N$ sufficiently large depending on $\eps_0$,   we obtain 
\begin{equation}\label{eqn:rapDec}
\| \sum_{T\in \T} \chi_T^N \|_{L^\infty(X)} \les R^{-2002}.
\end{equation}

\subsection{Proof of \Cref{thm:weightedL2}}
We define 
\[
    \|w\|_{L^2_{\mathrm{avg}}(T)}
    :=
    \left(
        \frac{1}{|T|}
        \int_{\mathbb R^3}w(x)^2 \chi_T(x)\,dx
    \right)^{1/2}.
\]
Let \(C(R)\) denote the best constant such that, for every
nonnegative weight \(w\) and every \(f\) satisfying
\(\supp\widehat f\subset N_{R^{-1}}(H)\),
\[
    \int_{\mathbb R^3}|f|^2w
    \leq
    C(R)
    \Bigg(
        \sum_{\theta\in\Theta(R^{-1})}
        \sum_{T\in\mathbb T_\theta}
        \|f_T\|_{L^2}^2
        \|w\|_{L^2_{\mathrm{avg}}(T)}^2
    \Bigg)^{1/2}
    \|f\|_{L^2}.
\]
To prove \Cref{thm:weightedL2}, it suffices to establish the following theorem.
\begin{theorem}\label{thm:weightedL2refined}
    For any $\eps>0$, there exists
$C_\eps>0$ such that
\begin{equation}\label{eqn:goalMain}
    C(R)\leq C_\eps R^\eps, \quad R\geq 1.
\end{equation}
\end{theorem}

We begin the proof of \Cref{thm:weightedL2refined}. To this end, we may assume that $R\in 4^\N$. We begin with a trivial bound.
\begin{lemma}\label{lem:trivialestimate}
    Let $\T\subset \cup_{\theta \in \Theta_{\mathrm{square}}(R^{-1})} \T_\theta$, $X\subset \R^3$ and $N\geq 1$. Then 
\begin{align*}
    \int_X \bigg|\sum_{T\in\mathbb T}f_T\bigg|^2w
    &\lesssim_N
    \| \sum_{T\in \T} \chi_T^N \|_{L^\infty(X)} \sum_{T\in \T}\frac{\|f_T\|_2^2}{|T|}\int_X w\chi_T \\
    &\lesssim_N \| \sum_{T\in \T} \chi_T^N \|_{L^\infty(X)}
    \left(
        \sum_{T\in\mathbb T}
        \|f_T\|_2^2\|w\|_{L^2_{\mathrm{avg}}(T)}^2
    \right)^{1/2}\bigg(\sum_{T\in\mathbb T} \|f_T\|_2^2\bigg)^{1/2}.
\end{align*}
\end{lemma}
\begin{proof}
Integrating \eqref{eq:separated-packet-pointwise} against $w\geq0$
gives
\[
    \int_X\bigg|\sum_{T\in\mathbb T}f_T\bigg|^2 w
    \lesssim_N
   \| \sum_{T\in \T} \chi_T^N \|_{L^\infty(X)} \sum_{T\in \T}\frac{\|f_T\|_2^2}{|T|}\int_X w\chi_T.
\]
For each $T$, the Cauchy--Schwarz inequality implies
\[
    \frac1{|T|}\int_X w\chi_T
    \lesssim \|w\|_{L^2_{\mathrm{avg}}(T)},
\]
where we used $\int_{X}\chi_T \leq \| \chi_T \|_{L^1} \lesssim|T|$.
Applying Cauchy--Schwarz to the sum over $T$, we obtain the second claimed inequality. 
\end{proof}

As a corollary, we show that $C(R)\les R$, which provides the base case for an induction argument.
\begin{lemma}\label{lem:base-case}
There exists a constant $C>0$ such that
\[
    C(R)\leq C R
\]
for every $R\geq1$.
\end{lemma}

\begin{proof}
Fix $R\geq1$. Recall the wave packet decomposition
\[
    f=\sum_{\theta\in \Theta_{\mathrm{square}}(R^{-1})}f_\theta= \sum_{\theta\in \Theta_{\mathrm{square}}(R^{-1})}\sum_{T\in \T_\theta} f_T.
\]
Using \Cref{lem:trivialestimate} with $X= \R^3$, $\T=\bigcup_{\theta\in\Theta_{\mathrm{square}}(R^{-1})}\T_\theta$ and $N=1$, as well as \eqref{eqn:trivialBound}, we obtain
\[
\begin{aligned}
    \int_{\mathbb R^3}|f|^2w
    &\lesssim
    R
    \left(
\sum_{\theta\in \Theta_{\mathrm{square}}(R^{-1})}\sum_{T\in \T_\theta}
        \|f_T\|_2^2\|w\|_{L^2_{\mathrm{avg}}(T)}^2
    \right)^{1/2}\left(\sum_{\theta\in \Theta_{\mathrm{square}}(R^{-1})}\sum_{T\in \T_\theta}\|f_T\|_2^2\right)^{1/2}.\\
    &\lesssim
    R 
    \Bigg(
        \sum_{ \theta\in\Theta_{\mathrm{square}}(R^{-1})}
        \sum_{T\in\mathbb T_\theta}
        \|f_T\|_{L^2}^2\|w\|_{L^2_{\mathrm{avg}}(T)}^2
    \Bigg)^{1/2}\|f\|_{L^2},
\end{aligned}
\]
where we used \eqref{eqn:orthogonalWavePacket} and Plancherel's theorem in the second inequality. Note that $\Theta_{\mathrm{square}}(R^{-1})\subset\Theta(R^{-1})$. Thus $C(R)\lesssim  R$ as claimed.
\end{proof}

The following estimate is the key ingredient in the proof of \Cref{thm:weightedL2refined}.

\begin{lemma}\label{lem:constant}
Let $R\in 4^\N$. For every \(\eps>0\), there exist \(A_\eps>0\) and an absolute
constant \(C_0\geq1\) such that, for every dyadic
\(K\in[1,R^{1/3}]\),
\[
    C(R)
    \leq
    C_0\left(
        C(K^{-2}R)+C(K^{-1}R)+K^6A_\eps R^\eps
    \right).
\]
\end{lemma}

Let $0<\eps<1/100$. Assuming \Cref{lem:constant}, we prove  \Cref{thm:weightedL2refined} by induction on the scale $R$. 

Choose \(R_\eps\) sufficiently large so that, whenever
\(R\geq R_\eps\), $K\geq R^{\eps/12}/2$ implies 
\(K\geq4\) and \(K^\eps\geq4C_0\). By \Cref{lem:base-case}, there is
\(\widetilde C_\eps>0\) such that
\[
    C(R)\leq\widetilde C_\eps
    \qquad\text{for }1\leq R\leq R_\eps.
\]
Set
\[
    C_\eps:=\max\bigl(\widetilde C_\eps,2C_0A_{\eps/2}\bigr).
\]
Then \(C(R)\leq C_\eps R^\eps\) for \(1\leq R\leq R_\eps\).

Now let \(R>R_\eps\) and assume the desired bound up to scale \(R/4\).
Choose a dyadic \(K\in[R^{\eps/12}/2,R^{\eps/12}]\). Since
\(K^{-2}R\leq K^{-1}R\leq R/4\), \Cref{lem:constant}, applied with
\(\eps/2\), and the induction hypothesis give
\[
    C(R)
    \leq
    C_0\left(
        C_\eps(K^{-2}R)^\eps
        +C_\eps(K^{-1}R)^\eps
        +K^6A_{\eps/2}R^{\eps/2}
    \right).
\]
Since \(K^6\leq R^{\eps/2}\) and
\(C_0K^{-\eps}\leq1/4\), we obtain
\[
    C(R)
    \leq
    \frac12C_\eps R^\eps
    +C_0A_{\eps/2}R^\eps
    \leq C_\eps R^\eps.
\]
This closes the induction and proves
\Cref{thm:weightedL2refined} subject to \Cref{lem:constant}.

\subsection{Broad--narrow reduction}
In this subsection, we begin the proof of  \Cref{lem:constant} with a broad--narrow reduction.
For a dyadic $K \in [1,R^{1/2}]$, we write 
\begin{align*}
    \Theta_{\text{strip}}(K^{-1}) = \Theta_{(K^{-1},1)} \cup \Theta_{(1,K^{-1})}. 
\end{align*}

\begin{definition}[$K^{-1}$-transverse rectangles]
Let $\delta=(\delta_1,\delta_2)$ satisfy $\delta_j\leq K^{-1}$ for $j=1,2$. We say that two rectangles $\tau_1,\tau_2\in\Theta_\delta$ are $K^{-1}$-transverse if each coordinate of $c(\tau_1)-c(\tau_2)$ has absolute value at least $2K^{-1}$.
\end{definition}

We use the following pointwise broad--narrow decomposition; see the proof of \cite[Theorem 2.4]{DW}.
\begin{proposition}\label{prop:broad-narrow}
For every $x\in\R^3$ and dyadic $K\in [1,R^{1/2}]$,
\[
|f(x)| \lesssim \max_{\tau\in \Theta_{\mathrm{square}}(K^{-2})  }|f_\tau(x)| +
\max_{S\in \Theta_{\mathrm{strip}}(K^{-1})}|f_{S}(x)| + K^2
\max_{\substack{\tau_1,\tau_2\in\Theta_{\mathrm{square}}(K^{-2})  \\
\tau_1,\tau_2\;K^{-1}\mathrm{-transverse}}}
|f_{\tau_1}(x)f_{\tau_2}(x)|^{1/2}.
\]
\end{proposition}

Consequently, we have
\begin{equation}\label{eqn:broadNarrow}
    \begin{split}
        \int_{\mathbb R^3} |f|^2 w \lesssim& \sum_{\tau\in \Theta_{\mathrm{square}}(K^{-2})  }\int_{\mathbb R^3}|f_\tau|^2 w +
\sum_{S\in \Theta_{\text{strip}}(K^{-1})} \int_{\mathbb R^3}|f_{S}|^2 w \\
&+ K^6
\max_{\substack{\tau_1,\tau_2\in\Theta_{\mathrm{square}}(K^{-2})  \\
\tau_1,\tau_2\;K^{-1}\text{-transverse}}}
\int_{\mathbb R^3} |f_{\tau_1}f_{\tau_2}|w.
    \end{split}
\end{equation}

We treat each term on the right-hand side of
\eqref{eqn:broadNarrow} using the following propositions.

\begin{proposition}[Scaling]\label{prop:scaling}
Let
\(\tau\in\Theta_{(K_1^{-1},K_2^{-1})}\), where
\(K_1,K_2\in[1,R^{1/2}]\) are dyadic. Then
\[
 \int_{\mathbb R^3}|f_\tau|^2w
 \leq 
 C(K_1^{-1}K_2^{-1}R)
 \Bigg( \sum_{\substack{
       \theta\in\Theta(R^{-1})\\
       \theta \subset \tau}}
 \sum_{T\in \T_\theta} \| f_T\|_{L^2}^2
   \|w\|_{L^2_{\mathrm{avg}}(T)}^2
 \Bigg)^{1/2}
 \|f_\tau\|_2
\]
whenever
\(\supp\widehat f\subset N_{R^{-1}}(H)\).
\end{proposition}

In particular, Plancherel's theorem gives
\[
    \sum_{\tau\in\Theta_{\mathrm{square}}(K^{-2})}
    \|f_\tau\|_2^2
    =
    \|f\|_2^2,
    \qquad
    \sum_{S\in\Theta_{\mathrm{strip}}(K^{-1})}
    \|f_S\|_2^2
    \lesssim
    \|f\|_2^2.
\]
Therefore,
\begin{equation}\label{eqn:scaling}
\begin{split}
 &\sum_{\tau\in\Theta_{\mathrm{square}}(K^{-2})}
   \int_{\mathbb R^3}|f_\tau|^2w
 +\sum_{S\in\Theta_{\mathrm{strip}}(K^{-1})}
   \int_{\mathbb R^3}|f_S|^2w\\
 &\qquad\lesssim
 \bigl(C(K^{-2}R)+C(K^{-1}R)\bigr)
  \Bigg( \sum_{
       \theta\in\Theta(R^{-1})}
 \sum_{T\in \T_\theta} \| f_T\|_{L^2}^2
    \|w\|_{L^2_{\mathrm{avg}}(T)}^2
 \Bigg)^{1/2}
 \|f\|_2,
\end{split}
\end{equation}
where we used Cauchy--Schwarz for the sums over $\tau$ and $S$.
For the bilinear term, we use the following estimate.

\begin{proposition}[Bilinear estimate]\label{prop:bilinear}
Let
\(\tau_1,\tau_2\in\Theta_{\mathrm{square}}(K^{-2})\) be
\(K^{-1}\)-transverse squares for some dyadic $K\in [1,R^{1/3}]$. For every \(\eps>0\), there exists
\(A_\eps>0\) such that, for every \(R\geq1\) and every nonnegative
weight \(w\),
\[
\begin{split}
 \int_{\mathbb R^3}|f_{\tau_1}f_{\tau_2}|w
 \leq{}&
 A_\eps R^\eps
 \prod_{j=1}^2
 \Bigg(
 \sum_{\substack{
       \theta\in\Theta(R^{-1})\\
       \theta\subset \tau_j}}
 \sum_{T\in \T_\theta} \| f_{T} \|_{L^2}^2
   \|w\|_{L^2_{\mathrm{avg}}(T)}^2
 \Bigg)^{1/4}
 \|f_{\tau_j}\|_2^{1/2}
\end{split}
\]
whenever
\(\supp\widehat f\subset N_{R^{-1}}(H)\).
\end{proposition}

Since \(\|f_{\tau_j}\|_2\leq\|f\|_2\) by Plancherel's theorem, 
\Cref{prop:bilinear} yields
\begin{equation}\label{eqn:bilinear}
 \int_{\mathbb R^3}|f_{\tau_1}f_{\tau_2}|w
 \leq
 A_\eps R^\eps
 \Bigg(
 \sum_{ \theta\in\Theta(R^{-1})}
 \sum_{T\in \T_\theta} \| f_{T} \|_{L^2}^2
   \|w\|_{L^2_{\mathrm{avg}}(T)}^2
 \Bigg)^{1/2}
 \|f\|_2
\end{equation}

Hence, \eqref{eqn:broadNarrow}, \eqref{eqn:scaling} and \eqref{eqn:bilinear} imply
\Cref{lem:constant}. We prove \Cref{prop:scaling} and \Cref{prop:bilinear} in the following
subsections.

\subsection{Rescaling: Proof of \Cref{prop:scaling}}

Let $\delta=(K_1^{-1},K_2^{-1})$.  Suppose that
\[
    \tau=\theta(c,\delta)
\]
for some  \(c=(c_1,c_2) \in (K_1^{-1}\Z \times K_2^{-1}\Z) \cap [-1,1)^2 \). Define \(g\) by
\[
    \widehat g(\zeta)   := (K_1K_2)^{-1}
    \widehat{f_\tau}(\mathcal A_\tau(\zeta)).
\]
By Plancherel's theorem, we have
\begin{equation}\label{eq:scaling-L2}
    \|g\|_2^2
    = \|f_\tau\|_2^2.
\end{equation}

A change of variables in the Fourier inversion formula gives
\[
    f_\tau(x)
    =
    (K_1K_2)^{-1}e^{2\pi ic_\tau \cdot x}
       g(L_\tau^{-1} x),
\]
for $c_\tau = (c_1,c_2,c_1c_2)$. Thus,
\begin{equation}\label{eq:scaling-physical}
    \int_{\mathbb R^3}|f_\tau|^2 w =
    \int_{\mathbb R^3}|g|^2\widetilde w, \text{ where }  \widetilde w:=w\circ L_\tau.
\end{equation}
Since $\supp \widehat{f_\tau} \subset N_{R^{-1}}(H_\tau)$, it follows from \Cref{lem:composition}, with $\theta'=[0,1]^2$, that
\[
    \supp\widehat g \subset \A_\tau^{-1} (N_{R^{-1}}(H_\tau)) = 
    N_{K_1K_2 R^{-1}}(H_{[0,1]^2}).
\]

We may now apply the definition of $C(K_1^{-1}K_2^{-1}R)$
to $g$ to obtain
\[
\int_{\mathbb R^3}|g|^2\widetilde w
\leq C(K_1^{-1}K_2^{-1}R)
\Bigg( \sum_{\theta'\in\Theta(K_1K_2R^{-1})}
\sum_{T'\in\mathbb T_{\theta'}}
\|g_{T'}\|_2^2
\|\widetilde w\|_{L^2_{\mathrm{avg}}(T')}^2 \Bigg)^{1/2} \|g\|_2.
\]
Note that if $\theta'$ is not a subset of $[0,1]^2$, then $g_{\theta'} = 0$ by the support property of $\hat{g}$. Hence, the sum over $\theta'\in\Theta(K_1K_2R^{-1})$ may be restricted to those $\theta'$ contained in $[0,1]^2$.

Let $\theta'\in\Theta(K_1K_2R^{-1})$ satisfy $\theta' \subset [0,1]^2$. Write $\theta'=\theta(d,\Delta)\in\Theta_\Delta$ for some $d\in (\Delta_1 \Z\times  \Delta_2 \Z)\cap [0,1)^2$ and $\Delta_1\Delta_2=K_1K_2R^{-1}$.
Then \[ \theta := \theta(c+\delta\circ d, \delta\circ \Delta) \in \Theta(R^{-1})\] as 
\[ c+\delta\circ d \in (K_1^{-1}\Delta_1 \Z \times K_2^{-1}\Delta_2 \Z)\cap [-1,1)^2, \quad \text{ and }\quad  (K_1^{-1} \Delta_1)(K_2^{-1} \Delta_2) =  R^{-1}.  \]
Moreover, since $d_i + \Delta_i \leq 1$, we have 
\[ \theta \subset \tau. \]
Recall from \Cref{lem:composition} that
$L_\tau L_{\theta'}=L_\theta$. Thus, for
$T'=T_{\theta',k}\in\T_{\theta'}$, we have
\[
    T:=L_\tau(T') = L_\tau L_{\theta'}(k/8+Q_1)
    =L_\theta(k/8+Q_1)
    =T_{\theta,k}\in\T_\theta.
\]
Since $\chi_T(L_\tau y)=\chi_{T'}(y)$, a change of variables gives
\[
    \|\widetilde w\|_{L^2_{\mathrm{avg}}(T')}^2
    =\|w\|_{L^2_{\mathrm{avg}}(T)}^2.
\]

We next compare the corresponding wave packets. 
The definition of $g$ and Fourier inversion give
\[
    f_\theta(x)
    =(K_1K_2)^{-1}e^{2\pi i c_\tau\cdot x}
    g_{\theta'}(L_\tau^{-1}x).
\]
Since
\[
    M_\theta=M_\tau M_{\theta'},
    \qquad
    c_\theta=c_\tau+M_\tau c_{\theta'},
\]
and
\[
    J_\theta^{1/2}
    =(\det M_\tau)^{1/2}J_{\theta'}^{1/2}
    =(K_1K_2)^{-1}J_{\theta'}^{1/2},
\]
the definition of $\psi_{\theta,k}$ gives
\[
    \psi_{\theta,k}(x)
    =(K_1K_2)^{-1}e^{2\pi i c_\tau\cdot x}
    \psi_{\theta',k}(L_\tau^{-1}x).
\]
Since $(K_1K_2)^{-2}\det L_\tau=1$, we obtain
\[
    \langle f_\theta,\psi_{\theta,k}\rangle
    =\langle g_{\theta'},\psi_{\theta',k}\rangle.
\]
Therefore, by the definition of the wave packets,
\[
    f_T(x)
    =(K_1K_2)^{-1}e^{2\pi i c_\tau\cdot x}
    g_{T'}(L_\tau^{-1}x).
\]
Another change of variables yields
\[
    \|f_T\|_2^2
    =(K_1K_2)^{-2}\det(L_\tau)\|g_{T'}\|_2^2
    =\|g_{T'}\|_2^2.
\]

Consequently,
\begin{equation}\label{eq:weight-scaling}
 \sum_{\substack{
       \theta'\in\Theta(K_1K_2R^{-1})}}
 \sum_{T'\in \T_{\theta'}} \| g_{T'}\|_{L^2}^2
    \|\widetilde w\|_{L^2_{\mathrm{avg}}(T')}^2
 \leq
 \sum_{\substack{
       \theta\in\Theta(R^{-1}) \\ \theta \subset \tau}}
 \sum_{T\in \T_\theta} \| f_{T}\|_{L^2}^2
    \|w\|_{L^2_{\mathrm{avg}}(T)}^2.
\end{equation}
Combining this with  \eqref{eq:scaling-L2} and \eqref{eq:scaling-physical} finishes the proof of  \Cref{prop:scaling}.

\subsection{Bilinear estimates}
In this section, we reduce \Cref{prop:bilinear} to a bilinear estimate involving $1/4$-transverse rectangles. 
\begin{proposition}\label{prop:bilinear-separated}
Let $\Delta = (\Delta_1,\Delta_2)$ be a pair of dyadic numbers such that $R^{-1/2} \leq \Delta_j \leq 1/4$ for $j=1,2$. Let $\sigma_1=\theta(0,\Delta)$ and $\sigma_2=\theta(c,\Delta)$ be rectangles in $\Theta_\Delta$ such that $|c_j|\geq 1/2$ for $j=1,2$. 
Then, for every \(\eps>0\) and every nonnegative weight \(w\),
\[
 \int_{\mathbb R^3}|f_{\sigma_1}f_{\sigma_2}|w
 \lesssim_\eps
 R^\eps
 \prod_{j=1}^2 \bigg( \sum_{\substack{
       \theta\in\Theta_{\mathrm{square}}(R^{-1})\\
       \theta\subset\sigma_j}}
 \sum_{T\in \T_\theta}  \| f_T\|_{L^2}^2
     \| w\|_{L^2_{\mathrm{avg}}(T)}^2 \bigg)^{1/4} \| f_{\sigma_j}\|_{L^2}^{1/2},
\]
whenever
\(\supp\widehat f\subset N_{R^{-1}}(H)\).
\end{proposition}
The proof of \Cref{prop:bilinear-separated} is deferred to \Cref{sec:bilinear}.

\begin{proof}[Proof of \Cref{prop:bilinear}]
Recall that \(\tau_1,\tau_2\in\Theta_{\mathrm{square}}(K^{-2})\) for some dyadic $K\in [1,R^{1/3}]$. Then $c(\tau_j) \in (K^{-1}\Z \times K^{-1}\Z)\cap [-1,1)^2$ for $j=1,2$.
Write
\[
    d=c(\tau_2)-c(\tau_1) \in (K^{-1}\Z \times K^{-1}\Z).
\]
Let $\delta=(\delta_1,\delta_2)$ be a pair of dyadic numbers such that 
\[ |d_j|<\delta_j\leq2|d_j|,\qquad j=1,2. \]
We set
\[ \Delta=\delta^{-1}\circ (K^{-1},K^{-1}).\] 
Then, for $j=1,2$,  
\[\Delta_j=\delta_j^{-1}K^{-1} \leq 1/4\]
as $\delta_j^{-1}K^{-1}$ is dyadic and $\delta_j^{-1}K^{-1} < |d_j|^{-1}K^{-1} \leq 1/2$, where the second inequality follows from the assumption that $\tau_1,\tau_2$ are $K^{-1}$-transverse.

Define $g$ by 
\[ \hat{g}(\zeta) = \delta_1\delta_2 \hat{f}\bigl(\A_\tau(\zeta)\bigr)1_{[-1,1]^2\times \R}(\zeta), \;\; \text{ where } \; \tau = \theta(c(\tau_1), \delta).\]
Then 
\[ \supp \hat{g} \subset N_{R_0^{-1}}(H), \;\; \text{ where } \;\; R_0 = \delta_1\delta_2 R \]
by \Cref{lem:composition}. Moreover,  
\[ \tau_j \times \R = \A_\tau (\sigma_j \times \R),  \]
where
\[ \sigma_j=\theta\big(\delta^{-1}\circ( c(\tau_j)-c(\tau_1)), \Delta \big).\]

Let 
\[ c:=c(\sigma_2) = \delta^{-1}\circ d \in \Delta_1\Z \times \Delta_2 \Z,\] 
so that 
\[ \sigma_1 = \theta(0,\Delta), \; \text{ and } \; \sigma_2 = \theta(c,\Delta).\]
Since $\delta_j\leq 2|d_j|$,
\[ |c_j| \geq 1/2, \;\;\; j=1,2. \]
Clearly, $\sigma_1 \in \Theta_{\Delta}$. Since $\delta_j>|d_j|$, we have  $c \in (-1,1)^2$, thus $\sigma_2 \in \Theta_{\Delta}$. 
In particular, $\sigma_j \subset [-1,1]^2$ for each $j=1,2$, which yields
\begin{align*}
\widehat{g_{\sigma_j}}(\zeta)
    =
   \widehat{g}(\zeta)  1_{\sigma_j \times \R}(\zeta) 
    = \delta_1\delta_2  \widehat{f}(\mathcal A_\tau(\zeta))  1_{\tau_j \times \R}\bigl(
    \A_\tau(\zeta)\bigr) = \delta_1\delta_2  \widehat{f_{\tau_j}}(\mathcal A_\tau(\zeta)). 
\end{align*}
Plancherel's theorem gives
\begin{equation}\label{eq:bilinear-affine-L2}
    \|f_{\tau_j}\|_2=
     \|g_{\sigma_j}\|_2.
\end{equation}
Moreover, by Fourier inversion, 
\[
    |f_{\tau_j}(x)|
    =
    \delta_1\delta_2 |g_{\sigma_j}(L_\tau^{-1} x)|.
\] 
Thus, a change of variables yields
\begin{equation}\label{eq:bilinear-affine-integral}
    \int_{\mathbb R^3}|f_{\tau_1}f_{\tau_2}|w
    =
    \int_{\mathbb R^3}|g_{\sigma_1}g_{\sigma_2}|\widetilde w, \text{ where }  \widetilde w:=w\circ L_\tau.
\end{equation}

We now apply \Cref{prop:bilinear-separated} to $g_{\sigma_1}$ and $g_{\sigma_2}$ at scale $R_0=\delta_1\delta_2R$. The only remaining hypothesis to check is $\Delta_j \geq R_0^{-1/2}$, or equivalently, $\delta_j^{-2}K^{-2} \delta_1\delta_2 R \geq 1$. Since $\delta_j \in (2K^{-1}, 2]$, we have
\[ \delta_j^{-2} K^{-2} \delta_1\delta_2 R \geq K^{-3} R \geq 1. \]
In the second inequality, we used the assumption that $K\leq R^{1/3}$. 
Thus, by \Cref{prop:bilinear-separated},  
\[
\begin{split}
\int_{\mathbb R^3}|g_{\sigma_1}g_{\sigma_2}|\widetilde w
\lesssim_\eps{}&
R_0^\eps
\prod_{j=1}^2
\Bigg(
\sum_{\substack{
\theta'\in\Theta_{\mathrm{square}}(R_0^{-1}) \\
\theta' \subset \sigma_j}}
\sum_{T'\in\mathbb T_{\theta'}}
\|g_{T'}\|_2^2
\|\widetilde w\|_{L^2_{\mathrm{avg}}(T')}^2
\Bigg)^{1/4}
\|g_{\sigma_j}\|_2^{1/2}.
\end{split}
\]
By applying \eqref{eq:weight-scaling}
(with $K_1K_2=(\delta_1\delta_2)^{-1}$),
\eqref{eq:bilinear-affine-L2} and
\eqref{eq:bilinear-affine-integral}, we obtain
\[
\begin{split}
\int_{\mathbb R^3}|f_{\tau_1}f_{\tau_2}|w
\lesssim_\eps{}&
R^\eps
\prod_{j=1}^2
\Bigg(
\sum_{\substack{
\theta\in\Theta(R^{-1}) \\ \theta\subset \tau_j
}}
\sum_{T\in\mathbb T_\theta}
\|f_T\|_2^2
\|w\|_{L^2_{\mathrm{avg}}(T)}^2
\Bigg)^{1/4}
\|f_{\tau_j}\|_2^{1/2}.
\end{split}
\] This proves \Cref{prop:bilinear}.

\end{proof}
Thus, \Cref{thm:weightedL2refined}, and hence \Cref{thm:weightedL2}, follows subject to the proof of \Cref{prop:bilinear-separated}.

\section{Proof of the bilinear estimate: \Cref{prop:bilinear-separated}}\label{sec:bilinear}
It remains to prove \Cref{prop:bilinear-separated}. 
We recall that functions $f$ satisfying $\supp\widehat f\subset N_{R^{-1}}(H)$ admit the wave packet decomposition 
\[ f = \sum_{\theta \in \Theta_{\text{square}}(R^{-1})} f_\theta = \sum_{\theta \in \Theta_{\text{square}}(R^{-1})} \sum_{T\in \T_\theta} f_T, \]
where $f_T$ is a wave packet essentially localized to $T$ in physical space and to $\bar{\theta}$ in  frequency space. 
Here
\(\mathbb T_\theta =\{L_\theta(k/8+Q_1):k\in\mathbb Z^3\}\)
is a collection of parallelepipeds covering $\mathbb R^3$
with $O(1)$ overlap. Each $T\in\mathbb T_\theta$ has dimensions
$R^{1/2}\times R^{1/2}\times R$. 
More generally, below we call a function $f_T$, indexed by $T\in\T_\theta$, a scale-$R$ wave packet if it is essentially localized to $T$ in physical space and to a fixed enlargement of $\bar{\theta}$ in frequency space.

We use a bilinear refined decoupling estimate for the hyperbolic paraboloid of Demeter and Wu \cite[Theorem 1.10]{DW}. Recall that, for $T\in \T_\theta$, we denote by $2T$ the concentric $R^{\eps_0}$-dilate of $T$ for small $\eps_0$ relative to the underlying $\eps>0$. For $\T \subset \bigcup_{\theta\in \Theta_{\text{square}}(R^{-1})} \T_\theta$ and a ball $Q$ of radius $R^{1/2}$, we let 
\[ M(Q,\T) = \# \{ T\in \T: 2T \cap Q \neq \emptyset \}. \]
In the following statement, we call a pair of squares $\Omega_1,\Omega_2 \subset [-1,1]^2$ \emph{transverse} if 
\[ |\xi_1-\xi_2|\sim |\eta_1-\eta_2| \sim 1, \; \text{ for all } \; (\xi_j,\eta_j)\in \Omega_j. \]

\begin{theorem}[Bilinear refined decoupling]
\label{lem:bilinear-refined-decoupling}
Let \(\Omega_1,\Omega_2 \subset [-1,1]^2\) be transverse squares.
For \(j=1,2\), let
\[
    f_j=\sum_{T\in\mathbb W_j}f_T
\]
be a sum of scale-\(R\) wave packets, where 
\[\W_j \subset \bigcup_{\substack{\theta \in \Theta_{\mathrm{square}}(R^{-1}) \\ \theta \subset \Omega_j}} \T_\theta.\]
Let $B_R$ denote a ball of radius $R$, and let $X\subset B_R$ be a union of pairwise disjoint
 balls \(Q\) of radius $R^{1/2}$.
Then, 
\begin{equation}\label{eq:bilinear-refined-decoupling}
 \int_X|f_1f_2|^2
 \less
 \prod_{j=1}^2
 \bigg( M_j
    \sum_{T\in\mathbb W_j}
    \|f_T\|_{L^4(\mathbb R^3)}^4
 \bigg)^{1/2},
\end{equation}
where  
\[ M_j = \max_{Q\subset X} M(Q,\W_j),\qquad j=1,2. \]
Here $A \less B$ denotes $A\leq C_\eps R^{\eps} B$  for some constant $C_\eps>0$ for any $\eps>0$.
\end{theorem}

Now we begin the proof of \Cref{prop:bilinear-separated}.
For $j=1,2$, take a scale-$R$ wave packet decomposition
\[
    f_{\sigma_j}
    =\sum_{\substack{
        \theta\in\Theta_{\mathrm{square}}(R^{-1})\\
        \theta\subset\sigma_j}}
      \sum_{T\in\T_\theta}f_T=\sum_{T\in\mathbb T_j}f_T, \; \text{ where }\; \T_j = \bigcup_{\substack{\theta \in \Theta_{\mathrm{square}}(R^{-1}) \\ \theta \subset \sigma_j}} \T_\theta.
\]
Thus, to prove \Cref{prop:bilinear-separated}, it suffices to show that 
\[
 \int_{\mathbb R^3}|f_{\sigma_1}f_{\sigma_2}|w
 \less 
 \prod_{j=1}^2 \bigg( 
 \sum_{T\in \T_j}  \| f_T\|_{L^2}^2
     \| w\|_{L^2_{\mathrm{avg}}(T)}^2 \bigg)^{1/4} \bigg( 
 \sum_{T\in \T_j}  \| f_T\|_{L^2}^2
      \bigg)^{1/4}
\]
since, by \eqref{eqn:orthogonalWavePacket} and Plancherel's theorem,
\[
    \sum_{T\in\mathbb T_j}\|f_T\|_2^2
    \sim
    \sum_{\substack{
        \theta\in\Theta_{\mathrm{square}}(R^{-1})\\
        \theta\subset\sigma_j}}
    \|f_\theta\|_2^2
    =\|f_{\sigma_j}\|_2^2.
\]

Cover $\mathbb R^3$ by a collection of balls $B$ of radius $R$ with bounded overlap. For each $B$, we set
\[
    \mathbb T_j(B)=\{T\in\mathbb T_j:2T\cap B\neq\emptyset\},
    \qquad
    f_{j,B}=\sum_{T\in\mathbb T_j(B)}f_T.
\]
Then 
\[ 
       \int_{\mathbb R^3}|f_{\sigma_1}f_{\sigma_2}|w \les \sum_B \Bigg( \int_{B}|f_{1,B}f_{2,B}|w +  \int_{B}|(f_{\sigma_1}-f_{1,B})f_{\sigma_2}|w + \int_{B}|f_{1,B}(f_{\sigma_2}-f_{2,B})|w \Bigg)
\]
by the identity
\[
    f_{\sigma_1}f_{\sigma_2}=f_{1,B}f_{2,B}+(f_{\sigma_1}-f_{1,B})f_{\sigma_2}
      +f_{1,B}(f_{\sigma_2}-f_{2,B}).
\]

The terms involving $f_{\sigma_1}-f_{1,B}$ and $f_{\sigma_2}-f_{2,B}$ are error terms and can be treated by trivial estimates. We verify this for the term with $f_{\sigma_1}-f_{1,B}$. By the Cauchy--Schwarz inequality,
\[     \sum_B\int_B
    |(f_{\sigma_1}-f_{1,B})f_{\sigma_2}|w \les \bigg(\sum_B\int_B
    |f_{\sigma_1}-f_{1,B}|^2 w \bigg)^{1/2} \bigg(\int_{\R^3} 
    |f_{\sigma_2}|^2 w \bigg)^{1/2}.\]
Since $f_{\sigma_1}-f_{1,B} = \sum_{T\in \T_1:\, 2T\cap B = \emptyset } f_T$,  \Cref{lem:trivialestimate} and \eqref{eqn:rapDec} yield 
\begin{align*}
    \sum_B\int_B
    |f_{\sigma_1}-f_{1,B}|^2w
    &\lesssim_\eps R^{-2002}
    \sum_{T\in \T_1}\|f_T\|_2^2\frac{1}{|T|}\int w\chi_T \\
    &\lesssim_\eps R^{-2002}
    \left(
        \sum_{T\in\mathbb T_1}
        \|f_T\|_2^2\|w\|_{L^2_{\mathrm{avg}}(T)}^2
    \right)^{1/2}\bigg(\sum_{T\in\mathbb T_1} \|f_T\|_2^2\bigg)^{1/2}.
\end{align*}
Together with the trivial bound for $\int_{\R^3} |f_{\sigma_2}|^2 w$ which follows from \Cref{lem:trivialestimate} and \eqref{eqn:trivialBound}, we obtain
\begin{align*}
    \sum_B\int_B
    |(f_{\sigma_1}-f_{1,B})f_{\sigma_2}|w  
    \lesssim_\eps R^{-1000}  \prod_{j=1}^2 \bigg( 
 \sum_{T\in \T_j}  \| f_T\|_{L^2}^2
     \| w\|_{L^2_{\mathrm{avg}}(T)}^2 \bigg)^{1/4} \bigg( 
 \sum_{T\in \T_j}  \| f_T\|_{L^2}^2
      \bigg)^{1/4}.
\end{align*}

Thus, it remains to control the main term involving $f_{1,B} f_{2,B}$. We claim that 
\begin{equation}\label{eq:separated-local-estimate}
    \int_B|f_{1,B}f_{2,B}|w
    \less\prod_{j=1}^2
    \Bigg(\sum_{T\in\mathbb T_j(B)}
        \|f_T\|_2^2\|w\|_{L^2_{\mathrm{avg}}(T)}^2
    \Bigg)^{1/4}  \Bigg(\sum_{T\in\mathbb T_j(B)}\|f_T\|_2^2\Bigg)^{1/4}.
\end{equation}
Applying Cauchy-Schwarz to the sum over $B$ in
\eqref{eq:separated-local-estimate}, we obtain
\[
    \sum_B\int_B|f_{1,B}f_{2,B}|w
    \less
    \prod_{j=1}^2
   \Bigg(\sum_{T\in\mathbb T_j}
        \|f_T\|_2^2\|w\|_{L^2_{\mathrm{avg}}(T)}^2
    \Bigg)^{1/4}  \Bigg(\sum_{T\in\mathbb T_j}\|f_T\|_2^2\Bigg)^{1/4}.
\]
where we used the fact that each $T\in \T_j$ belongs to $\T_j(B)$ for $O(R^{\eps_0}) \less 1$ balls $B$ in the covering. Thus, we have reduced the proof of \Cref{prop:bilinear-separated} to the local estimate \eqref{eq:separated-local-estimate}. 

Fix $B$. Since $\#\mathbb T_j(B)\lesssim R^3$, the contribution of packets
$T\in\mathbb T_j(B)$ for which
\[
    \|f_T\|_2
    \leq R^{-1000}
    \Bigg(\sum_{T\in\mathbb T_j(B)}\|f_T\|_2^2\Bigg)^{1/2}
\]
can be treated by \Cref{lem:trivialestimate}
and Cauchy--Schwarz. Thus, after dyadic pigeonholing and absorbing the resulting logarithmic losses into \(R^\eps\), it suffices to
consider a subcollection $\W_j \subset \mathbb T_j(B)$, satisfying
\[
    \|f_T\|_2\sim\nu_j \quad\text{for every }T\in\W_j
\]
for some dyadic $\nu_j>0$. Set
\[
    f_j=\sum_{T\in\W_j} f_T.
\]
Then it suffices to verify the local estimate \eqref{eq:separated-local-estimate} with $f_{j,B}$ replaced by $f_j$. 
We next use dyadic pigeonholing to reduce the local estimate on $B$ to an estimate on a union $X$ of balls of radius $R^{1/2}$, to which the bilinear refined decoupling theorem applies.

Cover $B$ by balls $Q$ of radius $R^{1/2}$ with bounded overlap, and
divide the covering into $O(1)$ collections of pairwise disjoint
balls. For every ball $Q$, let
\[
    m_j(Q):=M(Q,\W_j),
    \qquad w^2(Q):=\int_Qw^2.
\]
The contribution of the balls on which $m_1(Q)=0$
or $m_2(Q)=0$ is negligible. Indeed, let $X$ denote the union of balls $Q$ such that $m_j(Q)=0$ for a fixed $j$. Then every tube $T\in \W_j$ satisfies
$2T\cap X=\emptyset$. Thus,  by Cauchy-Schwarz, \Cref{lem:trivialestimate}, \eqref{eqn:trivialBound}, and \eqref{eqn:rapDec}, we have
\[
    \int_X |f_1f_2| w
    \lesssim_\eps R^{-1000}
    \prod_{j=1}^2
    \Bigg(\sum_{T\in\mathbb T_j(B)}
        \|f_T\|_2^2\|w\|_{L^2_{\mathrm{avg}}(T)}^2
    \Bigg)^{1/4}\Bigg(\sum_{T\in\mathbb T_j(B)}\|f_T\|_2^2\Bigg)^{1/4}.
\]

It remains to consider balls with $m_j(Q)\geq1$ for each $j=1,2$.
By dyadic pigeonholing in $m_1(Q)$ and $m_2(Q)$,  it suffices to consider a collection $\mathcal Q$
such that
\[
    m_j(Q)\sim M_j,\qquad j=1,2,
\]
for every $Q\in\mathcal Q$, where $M_1,M_2\geq1$ are fixed. Set
\[
    X:=\bigcup_{Q\in\mathcal Q}Q.
\]
Then by Cauchy--Schwarz,
\[
    \int_X|f_1f_2|w
    \leq
    \Bigg(\sum_{Q\in\mathcal Q}w^2(Q)\Bigg)^{1/2}
    \left(\int_X|f_1f_2|^2\right)^{1/2}.
\]
Next, we apply \Cref{lem:bilinear-refined-decoupling}. By the assumption on $\sigma_j$, there exist transverse squares $\Omega_1, \Omega_2 \subset [-1,1]^2$ such that $\Omega_j\supset \sigma_j$. By \Cref{lem:bilinear-refined-decoupling}, 
\begin{equation}\label{eq:separated-before-incidence}
    \int_X|f_1f_2|w
    \less
    \prod_{j=1}^2
    \Bigg(
        \sum_{Q\in\mathcal Q}w^2(Q) M_j
        \sum_{T\in\W_j}\|f_T\|_{L^4}^4
    \Bigg)^{1/4}.
\end{equation}

The Fourier support of $f_T$ is contained in a fixed enlargement of
a cap of dimensions $R^{-1/2}\times R^{-1/2}\times R^{-1}$. Therefore, by Bernstein's inequality,
\[
    \|f_T\|_4^4\lesssim |T|^{-1}\|f_T\|_2^4.
\]
By the choice of $\W_j$, it follows that
\[
    \sum_{T\in\W_j}\|f_T\|_4^4
    \lesssim |T|^{-1} \nu_j^2  \sum_{T\in \W_j}  \|f_T\|_2^2.
\]
For $j=1,2$, one has
\[
\sum_{Q\in\mathcal Q}w^2(Q)M_j\sim \sum_{Q\in\mathcal Q}w^2(Q)m_j(Q)=\sum_{T\in\W_j}
    \sum_{\substack{Q\in\mathcal Q\\2T\cap Q\neq\emptyset}} w^2(Q).
\]
Therefore, combining the above estimates, we obtain
\[ \sum_{Q\in\mathcal Q}w^2(Q)M_j  \sum_{T\in\W_j}\|f_T\|_4^4
\les   \Bigg( \sum_{T\in\W_j}
    \|f_T\|_2^2 |T|^{-1} \sum_{\substack{Q\in\mathcal Q\\2T\cap Q\neq\emptyset}} w^2(Q) \Bigg)\Bigg( \sum_{T\in \W_j}  \|f_T\|_2^2 \Bigg).\]

For each $T\in\W_j$,   balls $Q\in \mathcal Q$ intersecting $2T$ are contained in a fixed concentric enlargement $\tilde T$
of $2T$.  Thus, 
\[
|T|^{-1} \sum_{\substack{Q\in\mathcal Q\\2T\cap Q\neq\emptyset}} w^2(Q) \leq |T|^{-1}\int_{\tilde T} w^2 \less \|w\|_{L^2_{\mathrm{avg}}(T)}^2, \]
where we used $1_{\tilde T}\less \chi_T$. 
Substituting these estimates into \eqref{eq:separated-before-incidence}, we obtain
\[ 
    \int_X|f_1f_2|w
    \less
    \prod_{j=1}^2
      \Bigg( \sum_{T\in\W_j}
    \|f_T\|_2^2 \|w\|_{L^2_{\mathrm{avg}}(T)}^2 \Bigg)^{1/4}\Bigg( \sum_{T\in\mathbb \W_j}  \|f_T\|_2^2 \Bigg)^{1/4}.
\]
This proves the local estimate \eqref{eq:separated-local-estimate} and  hence \Cref{prop:bilinear-separated}.

\section{Proof of \Cref{thm:MT0} and \Cref{thm:MT}}
\label{sec:extension-equivalence}

\subsection{Proof of \Cref{thm:MT}}
It suffices to prove \eqref{eq:extension-MTthm}, since \eqref{main:eq2thm} then follows from \Cref{prop:lPerp}. Moreover, a standard argument based on the uncertainty principle shows that \eqref{eq:extension-MTthm} and \eqref{eqn:equivL2} are equivalent. We include a proof of this equivalence for completeness.

\begin{proposition}\label{prop:extension-equivalence}
Estimates \eqref{eq:extension-MTthm}  and  \eqref{eqn:equivL2} are equivalent.
\end{proposition}
\begin{proof}
We first verify the implication  \eqref{eqn:equivL2} $\implies$ 
\eqref{eq:extension-MTthm}. Choose
\(\psi\in C_c^\infty((-1,1))\) such that
\[
    |\check\psi(t)|\gtrsim1
    \qquad\text{whenever }|t|\leq1.
\]
Given \(g\in L^2([-1,1]^2)\), define \(F\) by
\[
    \widehat F(\xi,\eta,\tau)
    :=
    R\psi\bigl(R(\tau-\xi\eta)\bigr)g(\xi,\eta).
\]
Since \(\supp\psi\subset(-1,1)\), we have
\(\supp\widehat F\subset N_{R^{-1}}(H)\). Fourier inversion in the
\(\tau\)-variable gives
\[
    F(x)
    =
    \check\psi(x_3/R)Eg(x).
\]
For \(x\in B_R\), we have \(|x_3/R|\leq1\), and hence
\[
    |Eg(x)|\lesssim|F(x)|.
\]
Plancherel's theorem and the change of variables
\(s=R(\tau-\xi\eta)\) give
\[
\begin{aligned}
 \|F\|_2^2=\int
   R^2\left|\psi\bigl(R(\tau-\xi\eta)\bigr)\right|^2
   |g(\xi,\eta)|^2\,d\xi\,d\eta d\tau =R\|\psi\|_2^2\|g\|_2^2.
\end{aligned}
\]
Applying \eqref{eqn:equivL2} to \(F\), we obtain
\[
\begin{split}
\int_{B_R} |Eg|^2w \les \int |F|^2w \lesssim_\eps
 R^\eps
 \max_{\substack{
       \theta\in\Theta(R^{-1})\\
       F_\theta\neq0}}
 \sup_{T\parallel T_\theta}
 \left(
    \frac1{|T|}\int_Tw^2
 \right)^{1/2}
 \|F\|_2^2.
\end{split}
\]
Note that the condition \(F_\theta\neq0\) implies \(g_\theta\neq0\). Additionally, $|T|^{-1/2} \| F\|_{2}^2 \sim \| g\|_{2}^2$ since $|T|\sim R^2$ and $\| \psi\|_2\sim 1$.
Therefore,
\[
 \max_{\substack{
       \theta\in\Theta(R^{-1})\\
       F_\theta\neq0}}
 \sup_{T\parallel T_\theta}
 \left(
    \frac1{|T|}\int_Tw^2
 \right)^{1/2} \| F\|_{2}^2
 \les
 \max_{\substack{
       \theta\in\Theta(R^{-1})\\
       g_\theta\neq0}}
 \sup_{T\parallel T_\theta}
 \left(
   \int_Tw^2
 \right)^{1/2} \| g\|_{2}^2.
\]
Consequently, we obtain \eqref{eq:extension-MTthm}.

Conversely, assume \eqref{eq:extension-MTthm}, and let \(F\) satisfy
\(\supp\widehat F\subset N_{R^{-1}}(H)\). For
\(|s|\leq R^{-1}\), define
\[
    g_s(\xi,\eta)
    :=
    \widehat F(\xi,\eta,\xi\eta+s).
\]
Fourier inversion gives
\[
    F(x)
    =
    \int_{|s|\leq R^{-1}}
    e^{2\pi isx_3}Eg_s(x)\,ds.
\]
By Cauchy--Schwarz in \(s\),
\[
    |F(x)|^2
    \lesssim
    R^{-1}
    \int_{|s|\leq R^{-1}}|Eg_s(x)|^2\,ds.
\]
Integrating against \(w\) and applying Fubini's theorem, we get
\[
    \int_{B_R}|F|^2w
    \lesssim
    R^{-1}
    \int_{|s|\leq R^{-1}}
    \int_{B_R}|Eg_s(x)|^2w(x)\,dx\,ds.
\]
We apply \eqref{eq:extension-MTthm} to each \(g_s\). We first compare
the frequency supports of \(g_s\) and \(F\). Let
\(\theta=\theta(c,(\delta_1,\delta_2))\in\Theta(R^{-1})\), and
suppose that \((g_s)_\theta\neq0\). Then there exists $(\xi,\eta) \in \theta$ such that 
\[
    \widehat F(\xi,\eta,\xi\eta+s)\neq0.
\]
Hence \(F_\theta\neq 0\).
Therefore,
\[
 \max_{\substack{
       \theta\in\Theta(R^{-1})\\
       (g_s)_\theta\neq0}}
 \sup_{T\parallel T_\theta}
 \left(\int_Tw^2\right)^{1/2}
 \leq
 \max_{\substack{
       \theta\in\Theta(R^{-1})\\
       F_\theta\neq0}}
 \sup_{T\parallel T_\theta}
 \left(\int_Tw^2\right)^{1/2}.
\]
Applying \eqref{eq:extension-MTthm} to each \(g_s\), and using the
preceding bound, we obtain
\[
\begin{split}
 \int_{B_R}|F|^2w
 \lesssim_\eps{}R^\eps
 R^{-1}
 \max_{\substack{
       \theta\in\Theta(R^{-1})\\
       F_\theta\neq0}}
 \sup_{T\parallel T_\theta}
 \left(
    \int_Tw^2
 \right)^{1/2}
 \int_{|s|\leq R^{-1}}\|g_s\|_2^2\,ds.
\end{split}
\]
By the definition of \(g_s\),
\[
\begin{aligned}
 \int_{|s|\leq R^{-1}}\|g_s\|_2^2\,ds
 &=
 \int_{|s|\leq R^{-1}}
 \int_{[-1,1]^2}
 \left|
    \widehat F(\xi,\eta,\xi\eta+s)
 \right|^2
 \,d\xi\,d\eta ds\\
 &=
 \int_{\substack{
       (\xi,\eta)\in[-1,1]^2\\
       |\tau-\xi\eta|\leq R^{-1}}}
 |\widehat F(\xi,\eta,\tau)|^2
 \,d\xi\,d\eta d\tau=
 \|\widehat F\|_2^2
 =
 \|F\|_2^2.
\end{aligned}
\]
Here the third equality follows from
\(\supp\widehat F\subset N_{R^{-1}}(H)\), and the last equality
follows from Plancherel's theorem.
Finally, we recall that $|T| \sim R^2$. Combining these estimates, we obtain
\[
 \int_{B_R}|F|^2w
 \lesssim_\eps
 R^\eps
 \max_{\substack{
       \theta\in\Theta(R^{-1})\\
       F_\theta\neq0}}
 \sup_{T\parallel T_\theta}
 \left(
    \frac1{|T|}\int_Tw^2
 \right)^{1/2}
 \|F\|_2^2.
\]
The passage from this local estimate to the estimate on
\(\mathbb R^3\) is the same as at the end of the proof of
\Cref{prop:bilinear-separated}. This proves
\eqref{eqn:equivL2}.
\end{proof}

\subsection{Proof of \Cref{thm:MT0}}\label{sec:proofMT0}
As noted earlier, \Cref{thm:MT0} is essentially a consequence of \eqref{main:eq2thm}. We give a proof for completeness. 
Given \(g\in L^2([-1,1]^2)\), define $F$ as in the proof of \Cref{prop:extension-equivalence}. Since $\widehat{F}$ has compact support, we may write $F=F*\check{\chi}$ for some $\chi\in C_c^\infty(\mathbb R^3)$. 
Let \(\phi(x):=(1+|x|)^{-100}\), so that $|\check{\chi}|\les \phi$. By Cauchy--Schwarz,
\[
    |F|^2  \les |F|^2* \phi.
\]
Define
\[
    \widetilde w:=w*\phi.
\]
Using Fubini's theorem, we obtain
\begin{equation}\label{eq:extension-weight-smoothing}
    \int_{B_R}|Eg|^2w
    \lesssim
    \int_{\mathbb R^3}|F|^2\widetilde w.
\end{equation}

Thus, applying \Cref{thm:weightedL2} as in the proof of \Cref{prop:extension-equivalence}, we get 
\[
    \int_{B_R}|Eg|^2w
    \lesssim_\eps
    R^\eps
    \max_{\substack{
        \theta\in\Theta(R^{-1})\\
        g_\theta\neq0}}
    \sup_{T\parallel T_\theta}
    \left(
        \int_T \tilde{w}^2
    \right)^{1/2}
    \|g\|_2^2.
\]
By \Cref{prop:lPerp},
\[
\begin{split}
 \max_{\substack{
       \theta\in\Theta(R^{-1})\\
       g_\theta \neq0}}
 \sup_{T\parallel T_\theta}
 \left(
    \int_T\widetilde w^2
 \right)^{1/2}\lesssim
 R^{1/2}
 \sup_{\ell\perp\supp g}
 \left(
    \int_{N_1(\ell)}\widetilde w^2
 \right)^{1/2}.
\end{split}
\]
Since \(\widetilde w\) is roughly constant at scale \(1\), in the sense that $\widetilde w(x)\sim \widetilde w(y)$ whenever $|x-y|\leq 1$, 
\[
    \left(
       \int_{N_1(\ell)}\widetilde w^2
    \right)^{1/2}
    \leq
   \| \widetilde{w} \|_{L^\infty(N_1(\ell))}^{1/2} \left(
       \int_{N_1(\ell)}\widetilde w
    \right)^{1/2} \les \int_{N_1(\ell)}\widetilde w.
\]
Moreover, we have
\[ \int_{N_1(\ell)}\widetilde w = \int \int_{N_1(\ell)}  w(x-y) \phi(y) dx dy \les \sup_{\ell\perp \supp g} Xw(\ell). \] 

Combining the preceding estimates yields
\[
    \int_{B_R}|Eg|^2w 
    \lesssim_\eps
    R^{1/2+\eps}
    \sup_{\ell\perp\supp g}Xw(\ell)
    \|g\|_2^2,
\]
which, in particular, implies \Cref{thm:MT0}.

\bibliographystyle{plain}

\bibliography{references}

\end{document}